\documentclass[11pt]{amsart}
\usepackage{amsmath,amsthm,amsfonts,amssymb,latexsym,enumerate,graphicx,mathrsfs,mathabx,cancel}
\usepackage[utf8]{inputenc}
\usepackage[letterpaper]{geometry}
\usepackage{hyperref}
\hypersetup{colorlinks, citecolor=blue, linkcolor=red}
\usepackage{stmaryrd}
\usepackage{color}
\usepackage{tikz}
\usepackage{tikz-cd}
\usepackage[normalem]{ulem}
\usepackage{adjustbox}

\newtheorem{thm}{Theorem}[section]
\newtheorem{cor}[thm]{Corollary}
\newtheorem{prop}[thm]{Proposition} 
\newtheorem{lem}[thm]{Lemma}

\newtheorem{mainthm}{Theorem}

\theoremstyle{definition}
\newtheorem{defn}[thm]{Definition}

\newtheorem{question}[thm]{Question}

\theoremstyle{remark}
\newtheorem{remark}[thm]{Remark}
\makeatletter
\@tfor\next:=abcdefghijklmnopqrstuvwxyzABCDEFGHIJKLMNOPQRSTUVWXYZ\do{%
	\def\command@factory#1{%
		\expandafter\def\csname cal#1\endcsname{\mathcal{#1}}
		\expandafter\def\csname frak#1\endcsname{\mathfrak{#1}}
		\expandafter\def\csname scr#1\endcsname{\mathscr{#1}}
		\expandafter\def\csname bb#1\endcsname{\mathbb{#1}}
		\expandafter\def\csname rm#1\endcsname{\mathrm{#1}}
		\expandafter\def\csname bf#1\endcsname{\mathbf{#1}}
	}
	\expandafter\command@factory\next
}
\makeatother

\makeatletter
\DeclareFontEncoding{LS1}{}{}
\DeclareFontSubstitution{LS1}{stix}{m}{n}
\DeclareMathAlphabet{\mathscr}{LS1}{stixscr}{m}{n}
\makeatother
\newcommand{\define}{\emph}

\DeclareMathOperator{\aut}{\mathrm{Aut}}
\DeclareMathOperator{\out}{\mathrm{Out}}
\title[On the proofs of Leighton's Graph Covering Theorem]{On the proofs of Leighton's Graph Covering Theorem, a notion dual to commensurability, and normal virtual retracts}
\author[Nicholas Touikan]{Nicholas Touikan \\ with an appendix by Ashot Minasyan}

\begin{document}

\begin{abstract}
    Leighton's Graph Covering Theorem states that if two finite graphs have the same universal covering tree, then they also have a common finite degree cover. Bass and Kulkarni gave an alternative proof of this fact using tree lattices. We give an example of two graphs that admit a common finite cover which can not be obtained using tree lattice techniques. If two groups embed as finite index subgroups, we say they are co-commensurable. Our example comes from an explicit commensuration that cannot be induced by a co-commensuration. Next we state and prove a general theorem that gives necessary and sufficient conditions for when a commensuration can be induced by a co-commensuration. The developed machinery is then used to show that normal virtual retracts are virtual direct summands, answering a question of Merladet and Minasyan. In an appendix, applications to commensurating graphs of groups, biautomaticity, and hereditary conjugacy separability are given.
  \end{abstract}
\maketitle
\section{Introduction}

Leighton's Graph Covering Theorem \cite{leighton_finite_1982} asserts that if two finite graphs have the same universal covering tree then they admit a finite common covering space, answering a question posed in \cite{angluin_finite_1981}. Since, there have been various extensions of this result. There have been three proof strategies: constructing a finite cover by solving an integer programming problem (for example \cite{leighton_finite_1982,woodhouse_revisiting_2021}), using tree lattice techniques (for example \cite{bass_uniform_1990}), or using groupoids (for example \cite{shepherd_two_2022}). The theorem has also been generalized to decorated graphs and cube complexes \cite{woodhouse_revisiting_2021,shepherd_two_2022,woodhouse_leightons_2023,shepherd_commensurability_2024}. In \cite{neumann_leightons_2010} and \cite{shepherd_two_2022} different strategies are compared.

There have also been negative results giving the limitation to generalization such as \cite{bridson_leightons_2022} that states, for example, that if two finite graphs are covered by a common regular quasitree, this pair of coverings may not factor through a finite graph. Other ``non-Leighton'' examples have been constructed using non-positively curved square complexes \cite{wise_non-positively_1996,burger_lattices_2000} and there has also been work in finding minimal examples \cite{janzen_smallest_2009,dergacheva_small_2023,dergacheva_tiny_2025}. This paper is concerned with a different question:
\begin{question}\label{qu:one}
Given two finite graphs $X_1,X_2$ with the same universal covers, does the the tree-lattice approach in \cite{bass_uniform_1990} give all possible common finite covers of $X_1$ and $X_2$?  
\end{question}
 %It turns out that the answer is "no". Thus combinatorial approaches, such as those in \cite{leighton_finite_1982,woodhouse_revisiting_2021,shepherd_two_2022}, are more general.

We will now describe our approach to resolving this question, while also giving previously known results. Two groups $G_1,G_2$ are \define{commensurable} if they have a common finite index subgroup (up to isomorphism). In particular, any pair of topological spaces (that admit well-defined fundamental groups) that have a common finite degree covering space will have commensurable fundamental groups.

One way to show that two groups are commensurable is to show that they both embed as finite index subgroups in a common overgroup. This second property is dual to commensurability and we call it \define{co-commensurability}. This is the basis of the approach to Leighton's Theorem in \cite{bass_uniform_1990}: if $X_1,X_2$ are finite graphs with the common universal cover $T$ then actions by deck transformations of $\Gamma_1=\pi_1(X_1)$ and $\Gamma_2=\pi_1(X_2)$ on $T$ give embeddings $\Gamma_1,\Gamma_2 \leq \mathrm{Isom}(T)$. While there is no reason a-priori to expect $\Gamma_1,\Gamma_2$ to have non-trivial intersection in $\mathrm{Isom}(T)$, it is shown \cite[Theorem 4.7]{bass_uniform_1990} that there is a discrete subgroup $\Phi\leq\mathrm{Isom}(T)$ and $g_1,g_2\in\mathrm{Isom}(T)$ such that $[\Phi:\Gamma_i^{g_i}]<\infty, i=1,2$. This in turn implies $[\Gamma_i^{g_i}:\Gamma_1^{g_1}\cap\Gamma_2^{g_2}] <\infty, i=1,2$. Thus, any two such $\Gamma_1,\Gamma_2$ are co-commensurable and a common finite cover is obtained by taking the quotient $T/(\Gamma_1^{g_1}\cap\Gamma_2^{g_2})$. In \cite{neumann_leightons_2010}, so-called fat graphs are used to give discrete common finite index overgroups.

Co-commensurability easily implies commensurability and it is natural to ask whether the converse is true. It turns out that there are commensurable groups that cannot be embedded as finite index subgroups of a common overgroup.

The most well-known examples are non-conjugate maximal uniform lattices $\Gamma_1,\Gamma_2 \leq \mathrm{Isom}(\bbH^n)$, i.e. fundamental groups of hyperbolic orbifolds that cannot properly cover other orbifolds, that are commensurable in $\mathrm{Isom}(\bbH^n)$, i.e. $[\Gamma_i:\Gamma_1\cap\Gamma_2] <\infty, i=1,2$, but, by definition of maximal, cannot be both be contained in a common finite index discrete subgroup of $\mathrm{Isom}(\bbH^n)$. In other words, for any subgroup $\Gamma_1,\Gamma_2\leq H \leq \mathrm{Isom}(\bbH^n)$ we have $[H:\Gamma_i]=\infty$. This phenomenon can occur for so-called arithmetic lattices (see \cite{margulis_discrete_1991,maclachlan_arithmetic_2003}).

Now if we were able to find an ``abstract'' overgroup $\Gamma_1,\Gamma_2 \leq K$ that contained $\Gamma_1,\Gamma_2$ as finite index overgroups, then a result of Tukia \cite{tukia_quasiconformal_1986,tukia_convergence_1994} and Mostow rigidity \cite{mostow_quasi-conformal_1968} (see also \cite[Chapters 23-24]{drutu_geometric_2018} for a contemporary and unified account) imply that $K$ can actually be embedded (modulo a finite kernel) as a subgroup of $\mathrm{Isom}(\bbH^n)$ (in fact a uniform lattice) containing $\Gamma_1,\Gamma_2$ as proper finite index subgroups, contradicting maximality of $\Gamma_1,\Gamma_2$ as uniform lattices.

Commensurability therefore does not imply co-commensurability. Examples of groups that are commensurable but not co-commensurable show that the co-commensurability relation is not transitive (any group is trivially co-commensurable with its finite index subgroups).

While the question of (co-) commensurability within the class of free groups is easily settled (all finite rank nonabelian free groups embed as finite index subgroups of $F_2$), the more subtle question of whether a specific embedding of a group $H$ as a finite index subgroup of two free groups $F_1,F_2$ can be induced by an embedding into a common overgroup $K$, i.e. where we have a commutative diagram of inclusions as finite index subgroups
\[
    \begin{tikzcd}
        &H \arrow[dl,phantom,sloped,"\geq"] \arrow[dr,phantom,sloped,"\leq"]&\\
        F_1 \arrow[dr,phantom,sloped,"\leq"]& &F_2 \arrow[dl,phantom,sloped,"\geq"]\\
        &K&
    \end{tikzcd},
  \]
  was still unknown. It turns out that a necessary condition is that $H$ contains a normal subgroup $N$ that is itself simultaneously normal in both $F_1$ and $F_2$. This property in turn implies that the associated amalgamated free product $F_1*_HF_2$ admits a virtually free quotient and therefore also a finite quotient (see Lemma \ref{lem:coco_implies_finite_quotients}). The amalgamated free products of free groups constructed in \cite{bhattacharjee_constructing_1994}  give examples of free groups that can not be induced by a co-commensuration (see Corollary \ref{cor:incompletable1}.) This example still doesn't settle Question \ref{qu:one}. Using the more refined constructions in \cite{rattaggi_three_2007}, we are able to show that there is a finite graph $Z$ that covers graphs $X_1$ and $X_2$ such that the induced commensuration cannot be induced by a co-commensuration (see Theorem \ref{thm:no-BK}) thus:

  \begin{quote}
    \emph{The answer to Question \ref{qu:one} is ``no''.}
  \end{quote}

Question \ref{qu:one} was motivated by the author's attempts to generalize the approach in \cite{bass_uniform_1990} to other contexts in order to create common finite covers. Being able to do so has been a key step in recent quasi-isometric rigidity results of graphs of groups such as \cite{shepherd_quasi-isometric_2022,margolis_graphically_2023}. For example, although the preprint \cite{taam_quasi-isometric_2023} was withdrawn due to critical gap in the main argument, the authors (apparently correctly) applied the method of \cite{bass_uniform_1990} to show that if two groups $\Gamma_1,\Gamma_2$ acted freely, combinatorially, and cocompactly on a so-called churro-waffle space $X$ (i.e. they are uniform lattices in $\mathrm{Aut}(X)$) then can both be virtually embedded as finite index subgroups of a group $\Delta$. On the other hand, while Leighton's Theorem for ``graphs with fins'' was proved in \cite{woodhouse_revisiting_2021} using a linear programming and Haar measure approach, attempts to apply the method of \cite{bass_uniform_1990} have been unsuccessful to date. This negative answer to Question \ref{qu:one} gives some indication as to why this is the case.

To answer Question \ref{qu:one}, we rely on a simple condition to exclude co-commensuration and at this point it is natural to ask if this necessary condition, i.e. having a simultaneously finite index normal subgroup, is sufficient to construct a co-commensuration. It turns out there is another less obvious but nonetheless easy and natural condition we call \emph{out-finiteness} that is an additional necessary condition (see Proposition \ref{prop:condition2}). We then show that this additional condition is actually sufficient to construct a co-commensuration (see Theorem \ref{thm:converse}), which immediately gives:

\begin{mainthm}\label{thm:main}
  A commensuration between groups $G_1$ and $G_2$ over a group $H$ is induced by a co-commensuration if and only if $H$ contains a subgroup $N$ that embeds in both  $G_1$ and $G_2$ as a normal subgroup and such that corresponding commensuration over $N$ is out-finite.
\end{mainthm}

Out-finiteness, which is defined in Section \ref{sec:nec-suff}, is immediately satisfied if the group $N$ in the statement of Theorem \ref{thm:main} has finite outer automorphisms group. Thus, for example, if $\Gamma$ is a one-ended word hyperbolic group, by combining \cite[Theorem 1.4]{levitt_automorphisms_2005} and the main result of \cite{bowditch_cut_1998}, we have if the Gromov boundary of $\Gamma$ has no cutpairs, i.e. pairs of points whose removal disconnect, then $\out(\Gamma)$ must be finite and Theorem \ref{thm:main} immediately gives.

\begin{cor}\label{cor:com-hyp}
  If $\Gamma_1,\Gamma_2$ are word hyperbolic groups with connected and cutpair-free Gromov boundaries, then $\Gamma_1,\Gamma_2$ can be embedded as finite index subgroups of a common overgroup if and only if they have a common finite index normal subgroup.
\end{cor}

% Returning to our earlier hyperbolic manifold examples, noting that a uniform lattice in $\mathrm{Isom}(\bbH^n)$ will be quasi-isometric to $\bbH^n$ and therefore word hyperbolic group with Gromov boundary homeomorphic to the sphere $\bbS^{n-1}$, which has no cut pairs when $n\geq 3$, we can now make the following observation:

%\begin{cor}\label{cor:mflds}
%    Let $\Gamma_1,\Gamma_2 \leq \mathrm{Isom}(\bbH^n)$ be maximal uniform lattices (i.e. dual to minimal volume compact orbifolds) that are commensurable in $\mathrm{Isom}(\bbH^n)$ but non isomorphic, then no finite index subgroup $K\leq \Gamma_1\cap\Gamma_2\leq\mathrm{Isom}(\bbH^n)$ is simultaneously normal in $\Gamma_1$ and $\Gamma_2$.
%\end{cor}

We will also use our method to show that the non-trivial semidirect products $\bbZ^2\rtimes D_4$ and $\bbZ^2\rtimes D_6$ cannot be embedded into a common finite index over group (see Proposition \ref{prop:virt-z2-not-coco}.) The argument we give is elementary.

Here is an additional application of the techniques in this paper communicated by Ashot Minasyan. A subgroup $H \leq G$ is said to be a \emph{virtual retract} if there exists a finite index subgroup $K\leq G$ containing $H$ such that there is a map $\rho: K \twoheadrightarrow H$ that restricts to the identity on $H$. We have the following unexpected result.

\begin{mainthm}[{Normal virtual retracts have normal virtual complements (see \cite[Question 4.9]{minasyan_virtual_2025})}]\label{thm:normal_virt_comp}
    Let $G$ be a group with a normal  subgroup $N\lhd G$  such that $N$ is a virtual retract of $G$ and $G/N$ is finitely generated. Then for every finite index subgroup $H \leqslant G$ containing $N$
    there exists a finitely generated normal subgroup $M \lhd G$ such that $M \subseteq H$, $M\cap N=\{1\}$ and $|G:MN|<\infty$. In particular, $M$  is a normal virtual complement to $N$ in $G$, and the mapping $M\times N \to MN$ given by $(m,n)\mapsto mn$ is an isomorphism.
\end{mainthm}

Observe that in Theorem \ref{thm:normal_virt_comp}, $N \cong MN/M$ embeds as a finite index normal subgroup of $G/M$ and we have a natural injective homomorphism 
\[G \hookrightarrow G/N \times G/M.\] The image of $G$ under this homomorphism is subdirect (i.e., it projects onto each factor) and has finite index (see \cite[Lemma~2.1]{Min-cs_of_fibre_prods}). From this we deduce the following.

\begin{cor}\label{cor:virt-retr-subnormal}
    Let $N\lhd G$ be a normal virtual retract of a group $G$ such that $G/N$ is finitely generated. Then $G$ embeds as a finite index subdirect product in $\widetilde{N} \times G/N$, where $\widetilde{N}$ is a finite index supergroup of $N$.
\end{cor}

%Corollary \ref{cor:out-finite-virt-retract}, which is more technical, gives a new characterization of normal virtual retracts. 
This paper is essentially self-contained. After all, Theorem \ref{thm:main} is a theorem about general infinite groups, so one should not expect any specialized machinery. Although many arguments are informed by Bass-Serre theory, we only use the amalgamated free products that are naturally prescribed by a commensuration. The proof of Theorem \ref{thm:converse}, which involves arbitrary abelian groups, also uses elementary $\bbZ$-module theory.

\subsubsection*{Acknowledgments}
The author wishes to thank Alex Taam, Sam Shepherd and MathOverflow user HJRW (Henry Wilton) for useful discussions. The author also wishes to thank Adrien Le Boudec for diplomatically pointing out that one of the applications in a previous version of the paper was trivial, Ignat Soroko for providing useful feedback, and is extra grateful to Ashot Minasyan for, among other things, providing Theorem \ref{thm:normal_virt_comp}, discussing further applications of the techniques of this paper, for giving excellent suggestions to clarify the expostion, and for writing the appendix. The author is supported by an NSERC Discovery Grant.

Ashot Minasyan would like to thank the Isaac Newton Institute for Mathematical Sciences, Cambridge, for support and hospitality during the programme ``Actions on graphs and metric spaces'', where his work on the appendix to this paper was undertaken. This work was supported by EPSRC grant EP/Z000580/1.

\section{(co)commensuration}\label{sec:co-comm}

A \define{commensuration} is a pair of monomorphisms $(i_1,i_2)$ that have a common domain and whose images are finite index subgroups of their respective codomains. A commensuration $(i_1,i_2)$ is \define{trivial} if one of the monomorphisms is an isomorphism and \define{non-trivial} otherwise.  When we want to make the common domain of a commensuration explicit we will say that $(i_1,i_2)$ is a commensuration \define{over} $\mathrm{Dom}(i_1)=\mathrm{Dom}(i_2)$ and if we want to make the codomains explicit we will say it is a commensuration \define{between} $\mathrm{coDom}(i_1)$ and $\mathrm{coDom}(i_2)$. A \define{dual commensuration} or a \define{co-commensuration} is a pair of monomorphisms $(j_1,j_2)$ that have a common codomain and whose images are finite index subgroups of this common codomain. When we want to make the common codomain of a commensuration explicit we will say that $(j_1,j_2)$ is a co-commensuration \define{into} $\mathrm{coDom}(i_1)=\mathrm{coDom}(i_2)$.

Groups $G_1,G_2$ are said to be \define{commensurable} if there exists a commensuration ($i_1,i_2)$ between $G_1,G_2$, i.e. we have \begin{equation}\label{eqn:comm}
    G_1 \stackrel{i_1}{\hookleftarrow} H \stackrel{i_2}{\hookrightarrow} G_2,
\end{equation} where $i_1(H),i_2(H)$ have finite index in $G_1,G_2$ respectively. Although the domains $G_1,G_2$ and the codomain $H$ are already specified by the monomorphisms $i_1,i_2$ of a commensuration, it will be convenient for us (and equivalent) to refer to commensurations by diagrams such as the one given in \eqref{eqn:comm}.

We say $G_1,G_2$ are \define{co-commensurable} if they are the domains of the monomorphisms of a co-comensuration, i.e. they can both be embedded as finite index subgroups of a common overgroup. We say that a commensuration $(i_1,i_2)$ is \emph{normal} if the common domain of $i_1,i_2$ maps to normal subgroups of the codomains $G_1,G_2$ respectively. We say that a commensuration $(i_1,i_2)$ \emph{extends} a commensuration $(h_1,h_2)$ if we have commuting diagram\[
  \begin{tikzcd}
    &E \arrow[hookrightarrow]{dr}{h_2} \arrow[d,phantom,sloped,"\leq"]& \\
    G_1\arrow[hookleftarrow]{ru}{h_1}\arrow[hookleftarrow]{r}{i_1}  &H \arrow[hookrightarrow]{r}{i_2}  & G_2\\
  \end{tikzcd}
\] where $E$ is a finite index subgroup of $H$. The following result follows easily from the fact that if $G_1,G_2$ are finite index subgroups of a finitely generated group $K$, then their intersection contains a subgroup $H$ that is normal in $K$ and therefore in $G_1,G_2$.

\begin{prop}\label{prop:cocom_to_com}
    If $G_1,G_2$ are co-commensurable then there exists a normal commensuration $G_1 \stackrel{i_1}{\hookleftarrow} H \stackrel{i_2}{\hookrightarrow} G_2$ between $G_1$ and $G_2$.
\end{prop}
 
Thus co-commensurability implies commensurability, in fact it implies the existence of a normal commensuration. We now investigate the converse. A \define{completion} of a commensuration $(i_1,i_2)$ is a co-commensuration $(j_1,j_2)$ making the following diagram commute\begin{equation}\label{eqn:completion}
\begin{tikzcd}
    &H \arrow[hookrightarrow]{dr}{i_2}&\\
    G_1\arrow[hookleftarrow]{ru}{i_1}  \arrow[hookrightarrow]{dr}{j_1}&&G_2\\
    &K\arrow[hookleftarrow]{ru}{j_2} &\\
\end{tikzcd}.
\end{equation}

Trivial commensurations obviously admit completions.

\begin{lem}\label{lem:complete-extension}
  If a commensuration $(i_1,i_2)$ admits a completion then it extends a normal commensuration.
\end{lem}
\begin{proof}
  Suppose that the commensuration $(i_1,i_2)$ admitted a completion $(j_1,j_2)$ with common codomain $K$. Then $j_1\circ i_1 (H) = j_2\circ i_2 (H)$ is a finite index subgroup of $K$ and therefore contains a finite index normal subgroup $N$. $N$ is normal in $j_1(G_1)$ and $j_2(G_2)$ as well. Taking $E$ to be the subgroup the subgroup $E=i_1|_{i_1(H)}^{-1}\circ j_1|_{j_1(G_1)}^{-1}(N)\leq H$ gives the required normal commensuration. 
\end{proof}

Given a commensuration $G_1 \stackrel{i_1}{\hookleftarrow} H \stackrel{i_2}{\hookrightarrow} G_2$ we can form the \define{associated amalgamated free product} which we will denote as $G_1*_HG_2$. Given the commensuration $(i_1,i_2)$ the construction of the amalgamated free product is completely standard as a pushout, or fibered coproduct, in the category of groups.

Symmetrically, given an amalgamated free product $G_1*_HG_2$ where the image of the amalgamating subgroup is finite index in the factors $G_1$ and $G_2$, we can form the \define{associated commensuration}. It is worth emphasizing that although our amalgamated free product notation suppresses mention of the monomorphisms $i_1,i_2$, crucial properties of $G_1*_HG_2$ will depend not only on the triple of groups $G_1,H,G_2$ but also on the specific monomorphisms $i_1,i_2.$ We now give a first obstruction to completing a commensuration.

\begin{lem}\label{lem:coco_implies_finite_quotients}
    Let $G_1,G_2$ be finitely generated groups. If a non-trivial commensuration $G_1 \stackrel{i_1}{\hookleftarrow} H \stackrel{i_2}{\hookrightarrow} G_2$ admits a completion then the induced amalgamated free product $G_1*_HG_2$ admits an infinite virtually free quotient. In particular, $G_1*_HG_2$ has a nontrivial finite quotient.
\end{lem}

\begin{proof}
  By Lemma \ref{lem:complete-extension} $(i_1,i_2)$ extends a normal commensuration over some group $N\leq H$. Identifying $H$ with its canonical image in $G_1*_HG_2$, we get that $N \leq G_1*_HG_2$ is a finite index subgroup. Now since, by definition of a normal commensuration, $N$ is normal in both $G_1$ and $G_2$ it is normalized by a generating set of $G_1*_HG_2$ and is therefore normal in the entire amalgamated free product. Now it is easy to see that \[
(G_1*_HG_2)/N \simeq (G_1/N)*_{H/N}(G_2/N)
    \] which by \cite[Theorem 1]{karrass_finite_1973} is virtually free and therefore residually finite. In particular $G_1*_HG_2$ admits a non-trivial finite quotient.

%Let $T$ be the Bass-Serre tree dual to the amalgamated free product $G_1*_HG_2$. It is easy to see that every element of $N$ fixes every edge of $T$ and therefore that $N$ lies in the kernel of the action of $G_1*_HG_2$ on $T$. It therefore follows that $G_1*_HG_2/N$ also acts on $T$. On the one hand, we have\[
%G_1*_HG_2/N \simeq (G_1/N)*_{H/N} (G_2/N)
%    \] on the other hand since $N$ is finite index in both $G_1,G_2$ we get that $(G_1/N)*_{H/N} (G_2/N)$ is an amalgamated product of finite groups. Since the commensuration we were given was non-trivial, factors $G_1/H, G_2/H$ properly the amalgamating subgroup thus $(G_1/N)*_{H/N} (G_2/N)$ is an infinite virtually free group. Thus the main statement of the lemma holds.
%    Additionally, since virtually free groups are residually finite, $G_1*_HG_2$ will admit many finite quotients. 
\end{proof}

\begin{cor}\label{cor:incompletable1}
    There is a commensuration $F_3 \stackrel{i_1}{\hookleftarrow} F_6 \stackrel{i_2}{\hookrightarrow} F_3$ where $F_3,F_6$ denote the free groups of ranks 3,6 respectively that does not admit a completion.
\end{cor}
\begin{proof}
    In \cite{bhattacharjee_constructing_1994} an amalgamated free product $F_3*_{F_6}F_3$ is constructed with the following properties that, firstly, the subgroup $F_6$ embeds as a finite index subgroup of each factor groups and, secondly, that $F_3*_{F_6}F_3$ is \define{nearly simple}, which means that has not finite quotients. Therefore, by Lemma \ref{lem:coco_implies_finite_quotients}, the commensuration associated to this this amalgamated free product does not admit a completion.
\end{proof}

\section{An incompletable commensuration from finite degree common covers}
Any continuous function $f:(X,x) \to (Y,y)$ between path-connected based topological spaces that admit a fundamental group gives rise to a homomorphism $f_\sharp: \pi_1(X,x) \to \pi_1(Y,y)$ of fundamental groups. Conversely, any homomorphism between finitely generated groups can be realized by a continuous map between 2-complexes. It is not clear, however, whether the incompletable commensuration given in Corollary \ref{cor:incompletable1} can be realized by a pair of covering maps, which is a stronger geometric requirement. Such a pair can be found by taking a close look at the construction in \cite{rattaggi_three_2007}. 

\begin{thm}\label{thm:no-BK}
    There exist finite graphs $Z,X_1,X_2$ and finite degree covering maps $p_i:Z \to X_i, i=1,2$ such that the induced commensuration\[
        \pi_1(X_1) \stackrel{(p_1)_\sharp}{\hookleftarrow} \pi_1(Z) \stackrel{(p_2)_\sharp}{\hookrightarrow} \pi_1(X_2)
     \] does not admit a completion.
\end{thm}
\begin{proof}
  For this proof $F_n$ shall denote the free group of rank $n$. In \cite{rattaggi_three_2007} groups $\Lambda_1,\Lambda_3$ are constructed where $\Lambda_1$ is simple and $\Lambda_3$ has no finite quotients. Both groups decompose as amalgamated free products $F_9*_{F_{81}}F_9$ where the amalgamating subgroup $F_{81}$ embeds as a finite index in both $F_9$ factors. By Lemma \ref{lem:coco_implies_finite_quotients} both of these amalgamated products are associated to commensurations that do not admit completions. We will only consider $\Lambda_1$, the treatment of $\Lambda_3$ being identical.

  $\Lambda_1$ is the fundamental group of a 2-complex $\calX$ that is a degree 4 cover of a bouquet (or wedge product) of 10 circles to which 25 square 2-cells are attached. $\calX$ is a \emph{$\calV\calH$-complex}, i.e. a 2-complexes all of whose 2-cells are squares and whose edges can be partitioned into vertical and horizontal edges. This partition is obtained closing the relation ``two edges have the same orientation if they are on opposite sides of a square'' under reflexivity and transitivity to an equivalence relation. The 1-skeleton $\calX^{(1)}$ of $\calX$ (a degree 4 cover of the bouquet of 10 circles), as well as one of the 2-cells is shown in Figure~\ref{fig:sq-complex}.
  \begin{figure}[htb]
    \centering
    \includegraphics[width=\textwidth]{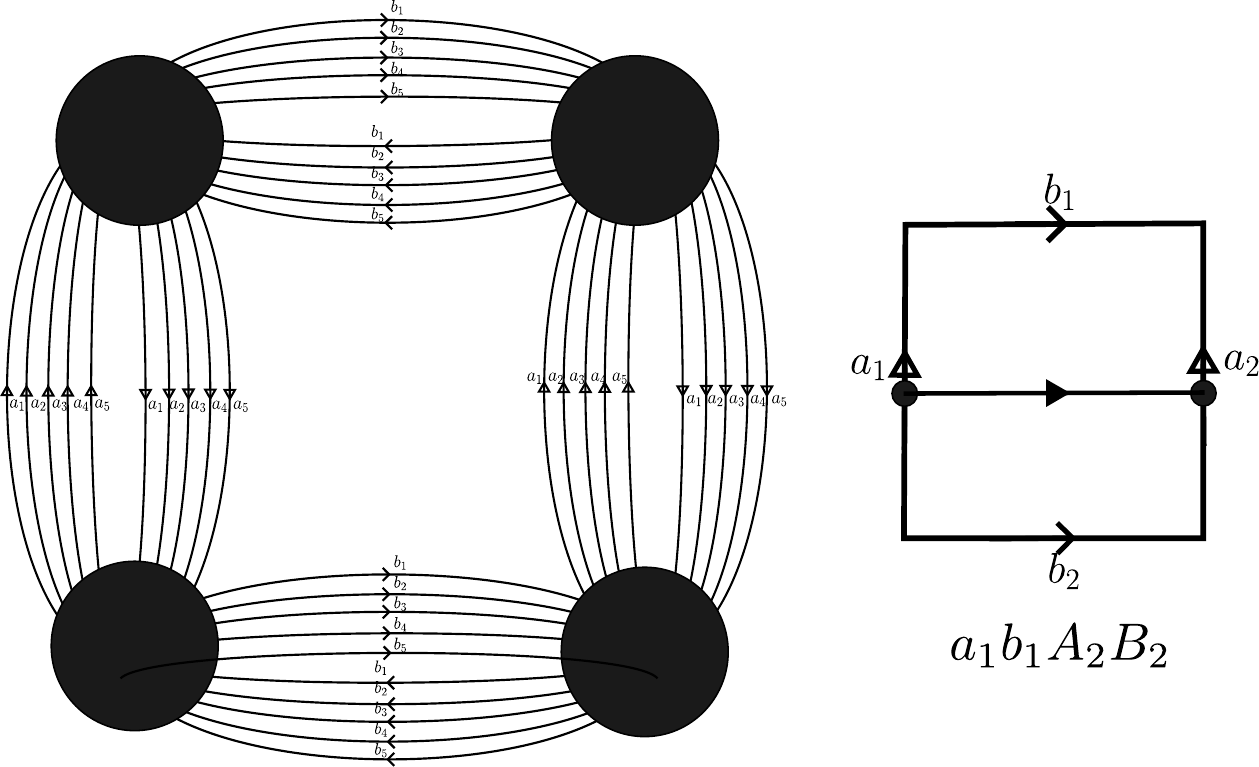}
    \caption{The labeled 1-skeleton for $\calX$ and a 2-cell corresponding to the relator $a_1b_1A_2B_2$. An edge and two vertices of the graph $Z$ are drawn in the square.}
    \label{fig:sq-complex}
  \end{figure}
  The graphs $X_1,X_2$ are respectively the top and bottom horizontal graphs shown in Figure \ref{fig:sq-complex} with directed edges labelled $b_1,\ldots,b_5$. Now the 2-complex $\calX$ has 100 squares, 4 for each of the relations given in \cite[Table 1]{rattaggi_three_2007}. The graph $Z \subset \calX$ is constructed as follows: its vertices are the midpoints of the vertical edges in $\calX^{(1)}$ and the edges of $Z$ are the segments in the squares obtained by joining these vertical edge midpoints. Such an edge is shown inside a square in Figure \ref{fig:sq-complex}. There is are mappings $p_i:Z \to X_i, i=1,2$ that can be defined on the edges of $Z$ as follows: an edge $e \in E(Z)$ lying in a square $\sigma$ is mapped to the top of $\sigma$, which lies in $X_1$ via $p_1$ and to the bottom edge in $X_2$ via $p_2$. These mappings on edges assemble to continuous maps $p_i:Z\to X_i$ (see \cite[\S 1.1]{wise_non-positively_1996} for more details and extended definitions.)

  Cutting $\calX$ along the graph $Z$ and then regluing expresses \[
    \pi_1(\calX) \simeq \pi_1(X_1)*_{\pi_1(Z)}\pi_1(X_2)\] by the Seifert-Van Kampen Theorem. The homomorphisms $\pi_1(Z) \hookrightarrow \pi(X_i)$ are given by $(p_i)_\sharp, i=1,2$. In particular in \cite{rattaggi_three_2007} these mappings are verified to have finite index images in their codomains. While it would be possible to go through the relations in \cite[Table 1]{rattaggi_three_2007} and verify that the combinatorial maps $p_i:Z\to X_i$ are indeed covering maps, we will give a less direct argument why this is true.

  First note that $Z$ has 20 vertices and 100 edges, therefore $\chi(Z)=-80$ so $\pi_1(Z) \simeq F_{81}$ which means that $p_i$ is $\pi_1$-injective. We also note that every vertex of $X_i$ has degree exactly 10. $p_i$ is also injective when restricted to edges, in fact $p_i$ is a combinatorial map. Thus if $p_i$ fails to be injective it will be at a vertex. We will first argue that $p_i$ must be locally injective.

  If $p_i$ isn't locally injective then will factor through a folding $Z \to Z'$ which is a surjective combinatorial map to another graph $Z'$ obtained by identifying two edges (see \cite{stallings_topology_1983}). We can repeatedly apply folding moves until we get $Z \to Z_F$ such that $p_i$ factors as $Z \to Z_F \to X_i$ and such that $Z_F \to X_i$ is locally injective combinatorial map. There are 100 edges in $Z$ which together contribute 200 to the sum of vertex degrees and there are 20 vertices giving an average degree of 10. Whenever a folding of edges occurs either two vertices get identified, or there were two edges with the same endpoints that get identified and the number of vertices is unchanged. In the latter case, a non-nullhomotopic cycle gets killed, which is impossible due to $\pi_1$-injectivity. This means that every folding move decreases the number of vertices by 1 and decreases the sum of the degrees by 2. Now the function\[
    A(f) = \frac{200-2f}{20-f}
  \] that gives the average degree of the vertices after $f$ folds is strictly increasing for $f\in[0,20)$. This means that $Z_F$ must have some vertex with degree greater than 10 contradicting the fact that the map $Z_F\to X_i$ is locally injective.

  It follows that $p_i$ must be locally injective, i.e. it does not factor through a folding map. Furthermore it must be locally surjective, otherwise this means it will have a vertex of degree less than 10, which because of the average degree of 10, implies that $Z$ must have some other vertex $w$ with degree more than 10, but then $p_i$ couldn't possibly be injective at $w$. Since $p_i,i=1,2$ are locally bijective combinatorial maps, we conclude that they are covering maps and this completes the proof.
\end{proof}

\section{Necessary and sufficient conditions to complete a commensuration}\label{sec:nec-suff}

So far we have been using Lemma \ref{lem:complete-extension} to produce commensurations that cannot be completed. At this point it is natural to ask whether the converse of Lemma \ref{lem:complete-extension} holds. Specifically, if given a commensuration $G_1 \stackrel{i_1}{\hookleftarrow} H \stackrel{i_2}{\hookrightarrow} G_2$ does the existence of a finite index subgroup $N\subset H$ such that the images $i_1(N)\leq G_1$ and $i_2(N)\leq G_2$ are normal in their respective overgroups imply that the commensuration is completable? In other words, does a commensuration along normal subgroups guarantee that there is a completion?

It turns out that there is another important necessary condition that we will now present. Recall that if $N\leq G$ is a normal subgroup then conjugation gives homomorphisms $\calA: G \to \aut(N)$ and $\calO:G/N \to \out(N)$, where $\mathrm{Inn}(N)$ denotes the group of inner automorphisms  and $\out(N) = \aut(N)/\mathrm{Inn}(N)$. In particular, we have the following commutative diagram
\[
\begin{tikzcd}
    G \arrow{r}{\calA}\arrow[two heads]{d} & \aut(N) \arrow[two heads]{d}\\
    G/N \arrow{r}{\calO} & \out(N)
\end{tikzcd}
\]
where the vertical arrows are the canonical quotient maps. We will abuse notation and also write $\calO:G \to \out(N)$ to mean the natural composition $G\to G/N \stackrel{\calO}\to \out(N)$. If $N\leq G$ is a finite index normal subgroup then the image of $G$, or $G/N$, in $\out(N)$ is finite. Thus given a normal commensuration $(i_1,i_2)$ between $G_1$ and $G_2$ over a group $H$, the groups $\calO(G_i/i_i(H)) \leq \out(H)$, must be finite for $i=1,2$. Our second necessary condition is given by the proposition below.

\begin{prop}\label{prop:condition2}
    If a normal commensuration $G_1 \stackrel{i_1}{\hookleftarrow} H \stackrel{i_2}{\hookrightarrow} G_2$  admits a completion then the subgroup\[
\langle \calO(G_1/i_1(H)),\calO(G_2/i_2(H)) \rangle \leq \out(H)    
    \] is finite.
\end{prop}
\begin{proof}
    By hypothesis our commensuration admits a completion $(j_1,j_2)$ as in \eqref{eqn:completion}. Replacing $K$ with $\langle j_1(G_1),j_2(G_2) \rangle$ if necessary and identifying $H,G_1,G_2$ with their images in $K$ we can assume that we have a subgroup inclusions\[
    \begin{tikzcd}
        &H \arrow[dl,phantom,sloped,"\geq"] \arrow[dr,phantom,sloped,"\leq"]&\\
        G_1 \arrow[dr,phantom,sloped,"\leq"]& &G_2 \arrow[dl,phantom,sloped,"\geq"]\\
        &K&
    \end{tikzcd}
  \] and that $K = \langle G_1,G_2\rangle.$ Since $H$ is normal in $G_1$ and $G_2$ we have that it is normal in $K$. Since $H\leq G_i \leq K$ we have a natural inclusions $\calO(G_i/H)\leq \calO(K/H), i=1,2$. Since $H$ has finite index in $K$ we have that $\calO(K/H) \geq \langle\calO(G_1/H),\calO(G_2/H)\rangle$ is finite.
\end{proof}

  Using the notation from Propostion \ref{prop:condition2} above, we say that a normal commensuration is \emph{out-finite} if the subgroup $\langle \calO(G_1/i_1(H)),\calO(G_2/i_2(H)) \rangle \leq \out(H)$ is finite.  It's not hard to see the following.

  \begin{lem}\label{lem:any-normal}
    If a commensuration admits a completion, then any normal commensuration it extends must be out-finite.
  \end{lem}

  While this next result is not difficult, it's worth recording for completeness, especially since no analogues come from the theory of lattices in $\mathrm{Isom}(\bbH^n)$.

\begin{prop}\label{prop:out-finite}
    There are normal commensurations that are not out-finite.
\end{prop}
\begin{proof}
    Let $G_1 = \bbZ^2 \rtimes D_4$ and $G_2 = \bbZ^2 \rtimes D_6$ where the semidirect factors $D_4,D_6$ act faithfully on $\bbZ^2$ normal factors. Consider any commensuration over the maximal $\bbZ^2$ factors of $G_1,G_2$. Now \[\aut(\bbZ^2)=\out(\bbZ^2) = \mathrm{GL}_2(\bbZ) \simeq D_4*_{D_2}D_6,\] (see \cite[\S I.5.2]{dicks_groups_1989}) where $D_n$ is dihedral symmetry group of the $n$-gon ($|D_n|=2n$). Since every finite subgroup of $\out(\bbZ^2)$ must be isomorphic to a subgroup of $D_4$ or $D_6$, and since $D_4$ has no elements of order 6 and $D_6$ has no elements of order $4$ the subgroup $\langle\calO(D_4),\calO(D_6)\rangle \leq \out(\bbZ^2)$ must be infinite.
\end{proof}

We can even get a stronger result:
\begin{prop}\label{prop:virt-z2-not-coco}
  The semidirect products $G_1=\bbZ^2 \rtimes D_4$ and $G_2=\bbZ^2\rtimes D_6$, where the actions of $D_4$ and $D_6$ are faithful, are not co-commensurable.
\end{prop}
\begin{proof}
  Suppose towards a contradiction that $G_1$ and $G_2$ were co-commensurable into a group $K$. Then by Proposition \ref{prop:cocom_to_com} there is a normal commensuration between $G_1$ and $G_2$ over some group $H$. $H$ has a characteristic subgroup $N$ that is isomorphic to $\bbZ^2$, so $N$ will also map to a normal subgroup of $G_1$ and $G_2$. Since $N$ maps to a finite index subgroup of the $\bbZ^2$ semidirect factors of $G_1$ and $G_2$, any linear non-trivial transformation of $\bbZ^2$ induced by conjugation will restrict to a non-trivial linear transformation of $N$ (otherwise it will fix a pair of linearly independent vectors) so the restrictions of the actions of $D_6$ and $D_4$ on $N$ will remain faithful so the argument of the proof of Proposition \ref{prop:out-finite} goes through and the result follows.
\end{proof}

\begin{question}\label{qu:virt-free-cocom}
  Are the groups $F_2\rtimes D_4$ and $F_2 \rtimes D_6$ co-commensurable? Since $\out(F_2) \simeq \out(\bbZ^2) \simeq \mathrm{GL}_2(\bbZ)$, the exact same argument of Proposition \ref{prop:out-finite} excludes the completion of the obvious commensuration over $F_2$. The argument of Proposition \ref{prop:virt-z2-not-coco} however doesn't go through since finite index subgroups of free groups become increasingly complicated.
\end{question}

We can now give, Theorem \ref{thm:converse}, the main technical positive result of this paper which is essentially a converse to Propositions \ref{prop:cocom_to_com},\ref{prop:condition2} and Lemma \ref{lem:any-normal}. Before proving stating and proving this theorem  we will fix some notation and give some auxiliary results. If $N,K$ are subgroups of some group $G$ we will write $NK = N\times K$ if $NK$ is a subgroup and the map $(n,k)\mapsto nk$ gives an isomorphism $N\times K \stackrel\sim\to NK$. Also, when convenient, we will express elements of $NK$ as pairs $(n,k)$ and that we identify $N,K$ with $N\times\{1\},\{1\}\times K$  respectively.

\begin{lem}[{Untwisting lemma (c.f. \cite[Lemma 5.1]{minasyan_virtual_2025})}]\label{lem:untwisting}
    Let $N\leq G$ be a normal subgroup such that the image $\calO(G/N) \leq \out(N)$ is finite  and let $J\leq G$ be a finitely generated free subgroup with $J\cap N = \{1\}$. Then there is a subgroup $K \leq NJ$ such that $K\cap N = \{1\}$, \[
    NK = N\times K,
    \] and the composition\[K \stackrel\sim\to NK/N \leq NJ/N \stackrel{\sim}{\to} J\] naturally maps $K$ to a finite index subgroup of $J$, so that $NK$ is a finite index subgroup of $NJ$.\end{lem}
\begin{proof}
  By hypothesis the image of $J$ in $\out(N)$ induced by conjugation, which is contained in $\calO(G/N)$, is finite. We may therefore take $K'$ to be a finite index subgroup of $J$ that lies inside the kernel $J\to \out(N).$ Considering the quotient map $NJ \to NJ/N \cong J$, it is clear that $NK' \leq NJ$ is a finite index subgroup.
  
  Pick a finite basis $\{k'_1,\ldots,k'_n\}$ of $K'$. Triviality of the image of $K'$ in $\out(N)$ tells us that for each $k'_i,i=1,\ldots,n$ there exists $h_i \in N$ such that for every $h \in N$ \[k'_i h {k'_i}^{-1} = h_i h {h_i}^{-1}.\]  We construct a new subgroup by modify the generating set, replacing $k'_i$ with $k'_ih_i^{-1} = k_i$, and taking $K' = \langle k_1,\ldots, k_n\rangle$. We have the equality $NK'=NK$ but now $K$ centralizes $N$. To see that $K\cap N = \{1\}$ take an arbitrary product such that\[
    k_{i_1}^{n_1}\cdots k_{i_l}^{n_l} \in N
  \] Since $N$ is normal we can expand $k_{i_j}=k'_{i_j}h_{i_j}$ and rewrite the product as\[
    k_{i_1}^{n_1}\cdots k_{i_l}^{n_l} =  (k_{i_1}')^{n_1}\cdots (k_{i_l}')^{n_l}h.
 \] $K'\cap N=\{1\}$, we must have that $(k_{i_1}')^{n_1}\cdots (k_{i_l}')^{n_l}$ is trivial and since $\{k'_1,\ldots,k'_n\}$ is a basis, this means that the original product $k_{i_1}^{n_1}\cdots k_{i_l}^{n_l}$ is trivial. Therefore, $NK = N\times K$ and the result follows.
\end{proof}

\begin{remark}
    Lemma \ref{lem:untwisting} does not hold if we drop the hypothesis that $J$ is free and allow $N$ to have non-trivial center. A counterexample is the group $N\ltimes \mathbb Z^2$ where $N=\langle x,y\rangle$ is free nilpotent of of rank 2 and class 3, and $\bbZ^2 = \langle a,b\rangle$ acts as $ana^{-1}=xnx^{-1}$ and $bnb^{-1} = yny^{-1}$ for all $n\in N$.
\end{remark}

The author thanks Ashot Minasyan for suggesting the following formulation of an intermediate step in an earlier proof of Theorem \ref{thm:converse} that is useful in its own right. The difficulties in this proof of this lemma comes from the fact that the centralizer $Z(N)$ need not be finitely generated and may contain torsion.

\begin{thm}[Normal virtual complement lemma]\label{thm:normal-in-normal}
  Let $N\leq G$ be a normal subgroup, let $K \leq G$ be a finitely generated group such that $NK = N\times K$, and $NK\leq G$ is a finite index normal subgroup. Then there is a subgroup $K_\Gamma \leq NK$ such that $K_\Gamma \leq G$ is normal in $G$ and $NK_\Gamma\leq G$ has finite index. Furthermore we have \[
    NK_\Gamma = N\times K_\Gamma
    \] and the composition \[
    K_\Gamma\stackrel\sim\to NK_\Gamma/N \leq NK/N \stackrel\sim\to K
    \] naturally embeds $K_\Gamma$ as finite index subgroup of $K$.
 \end{thm}

% \begin{remark}
%     This proof of this proposition is made more complicated when $Z(N)$, the center of $N$, is non-trivial and made even more complicated when $Z(N)$ has torsion. Since many interesting groups virtually have trivial centers, the reader may first want to consider the special case where $Z(N)$ is trivial, in which case Claim 1 is settled by a single paragraph. 
% \end{remark}
 \begin{proof}
   Since $NK \leq G$ is finite index and normal the image $\Gamma = \calO(G/NK) \leq \out(NK)$ is finite. We denote the conjugation maps by $G\ni g\mapsto \phi_g \in \aut(NK)$ and we denote by $\Phi_g \in \Gamma$ the image of the automorphism $\phi_g$ in $\out(NK)$.
   Since $N$ is normal in $G$, these automorphisms of $NK$ leave $N$ invariant, thus for any $g \in G$ and $k \in K$ we have \[
    \phi_g(k) = g k g^{-1} = k_gh_{k,g}
    \] for some $h_{k,g}\in N$ and $k_g\in K$. Since $g k g^{-1}$ must still commute with every element in $N$ we have that $h_{k,g} \in Z(N)$, i.e. it must be in the center of $N$. If $N$ has trivial center then $gkg^{-1} = k_g \in K \leq NK$ for all $k\in K, g\in G$ so $K$ is normal in $G$. In this case we set $K = K_\Gamma$ and the claim holds.

    %Otherwise, by hypothesis, $Z(N)$ is torsion-free. 
    Since $Z(N)$ is characteristic in $N$, the hypotheses imply that \[
        Z(N)K = Z(N)\times K
      \] is a normal subgroup of $G$. To continue, we must focus on the case where is some $k\in K$ and some $g \in G$ such that\[g k g^{-1} = \phi_g(k) = k_gh_{k,g} \not\in K.\]
      
    We will perform a sequence of modifications to $K$ to eventually get a normal subgroup $K_\Gamma$ of $G$  with the desired properties. On the one hand $K$ is finitely generated and $Z(N)$ is an abelian group, or equivalently a $\bbZ$-module. In particular, every automorphism of $Z(N)K$ descends to an automorphism of the abelianization \[
     \left(Z(N)K\right)_{ab}=Z(N)\times K_{ab} = M,
   \]
   where $K_{ab} = K/[K,K]$. For the $\bbZ$-module $M$ we will use additive notation and denote by $0$ the identity in $Z(N)$. We identify $K_{ab}$ with the submodule $\{0\}\times K_{ab}\leq M$.

   For any element of $g\in Z(N)K$ we will denote by $\bar g$ its image in $M$. Since $\aut(M)=\out(M)$ we actually have a natural action of the finite group $\Gamma$ on $M$, making $M$ into a $\bbZ\Gamma$-module. This structure is capable of detecting the non-normality of $K$ in $G$. Indeed, if $\phi_g(k) = k_gh_{k,g}\not\in K$ then $\Phi_g$ acts non-trivially on $M$ since \[
    \Phi_g \cdot(0,\bar k) = (\underbrace{h_{k,g}}_{\neq 0},\bar k_g) \in \left(Z(N) \times K_{ab}\right)\setminus \left(\{0\}\times K_{ab}\right).
    \] %Recall that $\Gamma \leq \out(\pi^{-1}(K))$ is the finite image of $G_1*_HG_2$ induced by the conjugation action and $\Gamma$ also acts naturally by automorphism on $M=Z(N)\times K_{ab}$, making $M$ into a $\Gamma$-module.\footnote{be explicit about the abelianization and}

    Let $\rho_0:Z(N)\times K_{ab} \twoheadrightarrow Z(N)$ be the canonical projection onto the first factor, by hypothesis $\rho_0$ is not $\Gamma$-invariant since its kernel $K_{ab}$ isn't. We now use a classical trick from the representation theory of finite groups (see \cite[\S 1.2]{fulton_representation_2004}). Let\[
    \rho(v) = \sum_{\gamma \in \Gamma} \gamma\cdot \rho_0(\gamma^{-1}\cdot v).
    \] By hypothesis, $Z(N)$ is $\Gamma$-invariant so $\rho: Z(N)\times K_{ab} \to Z(N)$ is a $\Gamma$-equivariant $\bbZ$-linear map to $Z(N)$ that restricts on $Z(N)$ to scalar multiplication by $|\Gamma|$. In particular $\rho$ is injective on set of infinite order elements of $Z(N)$ and we have\[Z(N)\cap\ker(\rho) = \{h \in Z(N): |\Gamma|h=0\}.\] $\Gamma$-equivariance of $\rho$ implies that $\ker(\rho)$ is also $\Gamma$ invariant. The goal is now to modify $K$ so that it maps to $\ker(\rho)$ and to get ``something like'' $M=Z(N)\times \ker(\rho)$. 

    There are two issues to overcome that stem from the fact that $Z(N)$ could be any abelian group. The first is that it is not clear from this construction that $M \leq Z(N)+\ker(\rho)$. The second issue is that even if $M = Z(N)+\ker(\rho)$, because the sumands may have non-trivial intersection, we could have $M\not\simeq Z(N)\times\ker(\pi)$. We will overcome these issues by passing to submodules.  
    %If $\mathrm{range}(\rho)=\rho(Z(N)) = |\Gamma|Z(N)$ then we would have that $M=Z(N)+\ker(\rho)$% which in turn would give us $M \simeq Z(N)\times \ker(\pi)$
    %. Since we may not assume this, we c
    Consider first %the submodule
    \[|\Gamma|M = \{\underbrace{m+\cdots+m}_{|\Gamma|\textrm{~terms}} \in M :m\in M\}.\] In this case, recalling that $\rho$ restricted to $Z(N)$ is scalar multiplication by $|\Gamma|$, since $\rho(|\Gamma|M) \leq \rho(Z(N))$ we can deduce the inclusion \begin{equation}\label{eqn:inclusion}
    |\Gamma|M\leq Z(N)+\ker(\rho)\end{equation} as follows: let $|\Gamma|m \in |\Gamma|M$ be arbitrary. Then $\rho(|\Gamma|m)=|\Gamma|\rho(m) = |\Gamma|h_m$ for some $h_m \in Z(N)$. Since $\rho(|\Gamma|m-h_m)=0$ we deduce that there is some $z\in \ker(\rho)$ such that $|\Gamma|m = h_m+z$ as required.

  Now $K_{ab}\cap |\Gamma|M = |\Gamma|K_{ab}$ is a finite index normal subgroup of $K_{ab}$ and so we define $K'_\Gamma \leq K$ to be its preimage. As a finite index subgroup of the finitely generated group $K$, $K'_{\Gamma}$ has a finite generating set $\{\kappa_1',\ldots,\kappa_m'\}$ and, since it's normal in $K$, its image in $K_\Gamma'/N \leq G/N$ is normal so $NK'_\Gamma$ remains normal in $G$ and we still have $NK_\Gamma =N\times K_\Gamma'\{1\}$. We denote by $\overline{K_\Gamma'}$ the image of $K'_\Gamma$ in the abelianization $M$. We note that $\overline{K_\Gamma'}$ will not necessarily be isomorphic to the abelianization of $K'_\Gamma$ as it will be a subgroup of $K'_{ab}\leq M$. $\overline{K'_\Gamma}$ lies in $Z(N)+\ker(\rho)$ so we have
    decompositions\[
        \overline{\kappa_i'} = h_i'+z_i', i=1,\ldots, m
      \] with $h_i' \in Z(N), z_i' \in \ker(\rho)$.

      In order to control torsion we will replace $M$ by a finitely generated module as follows: let 
    \begin{eqnarray*}
        M' &=& \mathrm{span}_{\bbZ\Gamma}\{h_1',\ldots,h_m',z_1',\ldots,z_m'\} \geq \overline{K_\Gamma'} \\
        M'' &=&  \mathrm{span}_{\bbZ\Gamma}M' \cup \rho(K'_{ab}).
    \end{eqnarray*}
    $M''$ is a finitely generated $\bbZ\Gamma$-module and, since $\Gamma$ is finite, it is also finitely generated as a $\bbZ$-module. It follows by the basic theory of finitely generated $\bbZ$-modules that there is a period $p \in \bbN$ such that $pM''$ is torsion-free.

    We repeat the construction above to obtain the finite index subgroup $K''_\Gamma \leq K'_\Gamma$ which is the preimage of \[p \overline{K_\Gamma'} = (p|\Gamma|)K_{ab}\leq pM''.\] Note that $K''_\Gamma \leq K$ is again normal and finite index. Let $k \in p\overline{K'_\Gamma}=\overline{K''_\Gamma}$. On the one hand we have\[
    \rho(k) = \rho(p|\Gamma|m) = p|\Gamma|\rho(m)=p|\Gamma|h
    \] for some $m\in K'_{ab}$ and some $h\in Z(N)\cap\rho(K'_{ab}) \leq Z(N)\cap M''$. In particular we find that \[
        k-ph =z\in \ker(\rho)
    \] and since $k,ph \in pM''$ we have that $z \in pM''$. From this we deduce \[
    \overline{K''_\Gamma}=(p|\Gamma|)K_{ab} \leq \left(Z(N)\cap pM''\right) + \left(\ker(\rho)\cap pM''\right).
    \] By eliminating torsion we have ensured that $\left(Z(N)\cap pM''\right)$ and $\left(\ker(\rho)\cap pM''\right)$ have trivial intersection. This means that if we take a generating set $\{\kappa''_1,\ldots,\kappa''_{r}\}$ of $K''_\Gamma$ and consider their images $\overline{\kappa''_i} \in (p|\Gamma|)K_{ab}\leq pM'', i=1,\ldots,r$ then for each $i$ we have a unique decomposition\[
        \overline{\kappa''_i} = h_i+z_i,
    \] with $h_i\in Z(N)\cap pM''$ and $z_i \in \ker(\rho)\cap pM''$.

    We finally take $K_\Gamma\leq Z(N)K_\Gamma''$ as the group with the generating set $\{\kappa_1,\ldots,\kappa_r\}$ where $\kappa_i = \kappa''_ih_{i}^{-1}$ and $\kappa''_i$ was a generator of $K_\Gamma''$ given above. On the one hand, we still have $Z(N)K''_\Gamma = Z(N)K_\Gamma$% and $Z(N)K_\Gamma = Z(N)\times K_\Gamma$
    . The abelianization map for $Z(H)K'$ restricted to $K_\Gamma$ still maps $K_\Gamma$ to $M''$ but now $K_\Gamma$ maps entirely to $\ker(\rho)\cap pM''$, which is $\Gamma$-invariant. 
    
    We will now show that $K_\Gamma$ is normal in $G$. Suppose otherwise, then there is some basis element $\kappa_i$ and some $g \in G$ with $g\kappa_i g^{-1} =\phi_g(\kappa_i) = k_{g,i}h_{g_i}$ with $k_{g,i}\in K_\Gamma, h_{g,i}\in Z(N)\setminus\{1\}$. Then mapping to $M$ we see that the image $\overline{\kappa_i} = z_i \in \ker(\rho)\cap pM''$ and that $\overline{\phi_g(\kappa_i)} = \Phi_g \cdot z_i = \overline{k_{g,i}}+h_{g,i}$. On the one hand $\ker(\rho)\cap pM''$ is $\Gamma$-invariant so $\overline{k_{g,i}}+h_{g,i} \in \ker(\rho)\cap pM''$. On the other hand, since $\overline{K_\Gamma} \leq \ker(\rho)\cap pM''$, we conclude that $h_{g,i} \in \ker(\rho)\cap pM''$. Now $Z(N)\cap \ker(\rho)$ consists of torsion elements, but $pM''$ is torsion-free which implies that $h_{g,i}=0$ contradicting the assumption that it is non-trivial.

    Thus, $K_\Gamma$ is a normal subgroup of $G$. Furthermore noting that the generators of $K_\Gamma$ were obtained by multiplying elements of $K''_\Gamma \leq K$ by elements of $Z(N)$ we have that $NK_\Gamma = NK''_\Gamma$. It remains to show that $K_\Gamma \cap N = \{1\}.$ First note that $K_\Gamma \leq C_G(N)$, the centralizer of $N$, thus $K_\Gamma \cap N \leq Z(N)$. For the image in the abelian group $M$ we have $\overline{K_\Gamma}\cap Z(N)=\{0\}$ therefore $K_\Gamma \cap Z(N) \leq [K,K]$, which is the kernel of the map $Z(N)K\twoheadrightarrow M$, and since $[K,K] \cap Z(N) = \{1\}$ we conclude $ K_\Gamma \cap N = K_\Gamma \cap Z(N) = \{1\}$. Therefore \[
      NK_\Gamma = N\times K_\Gamma,
    \] as required. In particular the natural embedding of $K_\Gamma$ as finite index normal subgroup of $K$ follows.
 \end{proof}

\begin{thm}\label{thm:converse}
    %Let $H$ virtually have a torsion-free center then 
    %Any out-finite normal commensuration admits a completion.\footnote{make for extension}
  Suppose a commensuration $(i_1,i_2)$ over a group $H$ extends a normal commensuration $(h_1,h_2)$ over some group $N$. If $(h_1,h_2)$ is out-finite then  $(i_1,i_2)$ admits a completion.
\end{thm}
 
\begin{proof}
Without loss of generality we can assume the commensuration is non-trivial, otherwise the result holds immediately. By hypothesis we have the following diagram of finite index inclusions with the image of $N$ being normal in all codomains.
\[
  \begin{tikzcd}
    &N \arrow[hookrightarrow]{dr}{h_2} \arrow[d,phantom,sloped,"\leq"]& \\
    G_1\arrow[hookleftarrow]{ru}{h_1}\arrow[hookleftarrow]{r}{i_1}  &H \arrow[hookrightarrow]{r}{i_2}  & G_2\\
  \end{tikzcd}
\] 
We form the associated amalgamated free product $G_1*_HG_2$ and we consider $N$ and $H$ as subgroups of $G_1$ and $G_2$ and $G_1,G_2$ as subgroups of $G_1*_HG_2$. By hypothesis, $N$ is a normal subgroup of this amalgamated free product and the canonical quotient is the free product\[
    (G_1*_HG_2)/N \simeq (G_1/N) *_{H/N} (G_2/N),    
    \] which, as an essential amalgamated product of finite groups, is virtually free by \cite[Theorem 1]{karrass_finite_1973}. This means that there is a finite index normal free subgroup $J \leq (G_1/N) *_{H/N} (G_2/N)$. Consider the commuting following commutative diagram of short exact sequences
    \begin{equation}\label{eqn:diag1}
        \begin{tikzcd}
            & & & 1\arrow{d} &\\
           1\arrow{r} & H \arrow[hook]{r}& \pi^{-1}(J) \arrow[two heads]{r}{\pi|_{\pi^{-1}(J)}} \arrow[phantom,sloped]{d}{\leq} & J \arrow{r}\arrow[hook]{d}& 1\\
           1 \arrow{r} & H \arrow[hook]{r} & G_1*_HG_2 \arrow[two heads]{r}{\pi}& (G_1/N)*_{H/N}(G_2/N) \arrow{r}1\arrow[two heads]{d}& 1\\
           & & & F\arrow{d}&\\
           & & & 1&\\
        \end{tikzcd}
    \end{equation}
    where hooked inclusion arrows ($\hookrightarrow$) are inclusions of normal subgroups, two headed arrows ($\twoheadrightarrow$) are canonical quotient maps, and $F$ is a finite group. Now since $J$ is a free group, the top row of the diagram splits and we have the semidirect product\[
        \pi^{-1}(J) = N\hat J \simeq N \rtimes_{\phi} \hat J   
      \] where $\hat J$ is an isomorphic lift of $J$ to $\pi^{-1}(J) \leq G_1*_HG_2$ and $\phi: \hat J \to \aut(N)$ is the homomorphism induced by conjugation of $\hat J$ on $H$. Now by hypothesis \[\calO((G_1*_HG_2)/N) = \langle\calO(G_1/N),\calO(G_2/N)\rangle \subset \out(N)\] is finite, so Lemma \ref{lem:untwisting} applies and we can find subgroup $K' \in N\hat J$ such that $NK' = N\times K'$ and such that $K'$ maps isomorphically via $\pi$ to a finite index subgroup of $J$. Now there is a finite index characteristic subgroup $\overline K \leq J$ such that $\overline K \leq \pi(K')$. Let $K = {\pi|_{K'}}^{-1}(\overline K)\leq K'$, then $KN = N\times K$ is the $\pi$-preimage of the normal subgroup $\overline K$ and is therefore a normal subgroup of $G_1*_HG_2$.

      This, however, is not enough to guarantee that $K$ is normal in $G_1*_HG_2$. At this point, however, we can apply Theorem \ref{thm:normal-in-normal} which gives a normal subgroup $K_\Gamma \leq NK$ that is also normal in $G_1*_HG_2$ and that maps via $\pi$ to finite index subgroup of $\pi(K)\leq J$. Since $K_\Gamma$ is normal in $G_1*_HG_2$, its image is normal in the quotient $(G_1/N)*_{H/N}(G_2/N)$. 

    We will now show that the desired completion is
\begin{equation}\label{eqn:second-completion}
\begin{tikzcd}
    &H \arrow[hookrightarrow]{dr}{i_2}&\\
    G_1\arrow[hookleftarrow]{ru}{i_1}  \arrow[hookrightarrow]{dr}{j_1}&&G_2\\
    &(G_1*_{H}G_2)/K_\Gamma \arrow[hookleftarrow]{ru}{j_2} &\\
\end{tikzcd}.
\end{equation}
First, looking at diagram \eqref{eqn:diag1}, since $\pi(K_\Gamma) \leq J$ which is torsion-free, its image has trivial intersection with the images of $G_1$ and $G_2$. So, since $K_\Gamma\cap N=\{1\}$, we have that $G_i \cap K_\Gamma = \{1\}, i=1,2$. Thus $G_1,G_2$ are mapped injectively via the canonical quotient map $G_1*_HG_2 \twoheadrightarrow (G_1*_HG_2)/K_\Gamma$. We can therefore identify $N,H,G_1,G_2$ with their images in $(G_1*_HG_2)/K_\Gamma$. Finally since $\pi(K_\Gamma)$ is a finite index normal subgroup of ${(G_1/N)*_{H/N}(G_2/N)}$ if we make an analogue of diagram \eqref{eqn:diag1} with $\pi(K_\Gamma)$ in place of $J$. It is immediate that \[
    ((G_1*_{H}G_2)/K_\Gamma)/N\simeq (G_1*_{H}G_2)/\pi^{-1}(\pi(K_\Gamma)) \simeq ((G_1/N)*_{H/N}(G_2/N))/\pi(K_\Gamma)
    \]  is finite. So $N$, and therefore also $G_1,G_2$, are finite index subgroups of $(G_1*_{H}G_2)/K_\Gamma$. Thus $(j_1,j_2)$ is the desired co-commensuration and the result follows.    
  \end{proof}

Next, we answer \cite[Question 4.9]{minasyan_virtual_2025}.

  \begin{proof}[Proof of Theorem \ref{thm:normal_virt_comp}]
    Since $N\lhd G$ is a virtual retract of $G$, there is a finite index subgroup $G'$ in $G$  such that there is a retraction $G' \twoheadrightarrow N$. The subgroup\[
      G''=\bigcap_{g\in G} g^{-1}(G' \cap H)g
    \] is a finite index normal subgroup of $G$ containing $N$ and contained in $G' \cap H$; in particular, it also retracts onto $N$. Thus there is a normal subgroup $K \lhd G''$ such that $K \cap N=\{1\}$ and $G''=KN$. Since $K$ and $N$ are both normal in $G''$ and intersect trivially, $G''$ is the internal direct product $K\times N$. 
    
    Note that $K \cong KN/N$ embeds as a finite index subgroup in $G/N$, hence $K$ is finitely generated. Therefore,   
we can apply Theorem \ref{thm:normal-in-normal} to obtain a finitely generated subgroup $M\leq KN$ where $M$ is normal in $G$, $MN = N\times K_\Gamma$ and $K_\Gamma N \leq G$ is a finite index subgroup. Thus $K_\Gamma$ is normal virtual complement of the normal virtual retract $N$ with the desired properties.
  \end{proof}

% Finally we strengthen the characterization of virtual retracts given by \cite[Lemma 4.5 and Lemma 5.1]{minasyan_virtual_2025}. We say a short exact sequence \[
% 1 \to N \to G \stackrel\pi\to G/N \to 1
% \]\emph{virtually splits} if there is a finite index subgroup $J\leq G/N$ and a homomorphism $\sigma:J \to G$ such that $\pi\circ\sigma=Id_J$.

%\begin{cor}\label{cor:out-finite-virt-retract}
%    A normal subgroup $N\leq G$ is a virtual retract of $G$ if and only if the short exact sequence\begin{equation}\label{eqn:split}
%    1 \to N \to G \stackrel\pi\to G/N \to 1
%    \end{equation} virtually splits and the natural image $\calO(G/N)\leq \out(N)$ induced by conjugation is finite.
%\end{cor}
%We next prove our characterization of normal virtual retracts. 
%\begin{proof}%[Proof of Corollary \ref{cor:out-finite-virt-retract}]
%      $(\Rightarrow)$ If $N\leq G$ is a normal virtual  retract then \eqref{eqn:split} virtually splits by Theorem~\ref{thm:normal_virt_comp} and the finiteness of $\calO(G/N)$ is stated in \cite[Lemma 4.5]{minasyan_virtual_2025}. 
      
%      $(\Leftarrow)$ This follows immediately from taking $\sigma(J)$, the finite index subgroup of $G/N$ that $\pi$-lifts, and applying Lemma \ref{lem:untwisting} and Theorem \ref{thm:normal-in-normal}.\comment{Ashot: Theorem~\ref{thm:normal-in-normal} is not needed here, so I don't think that Corollary~\ref{cor:out-finite-virt-retract} is new (or difficult).} Indeed, the finite index subgroup $NK_\Gamma = N \times K_\Gamma \leq G$ retracts onto $N$.
%\end{proof}   

\appendix
\section{Normal virtual complements to normal virtual retracts}
\label{sec:appendix}

\smallskip
\begin{center}by \textsc{Ashot Minasyan}\end{center}
\medskip
In this appendix we give an alternative proof of Theorem~\ref{thm:normal-in-normal} (see Proposition~\ref{prop:K'}), which quickly implies the statement of Theorem~\ref{thm:normal_virt_comp}, as seen  in Section~\ref{sec:nec-suff}. Our argument 
uses a lemma about finitely generated virtually abelian groups from \cite{Min-virt_retr_props}. We then discuss some applications of Theorem~\ref{thm:normal_virt_comp} to commensurating graphs of groups. 

\subsection{A shorter proof of Theorem~\ref{thm:normal-in-normal}.}
\begin{prop}\label{prop:K'}
Suppose a group $G$ has a finite index normal subgroup $G'$ such that
\begin{equation}\label{eq:G'}
G' = N \times K,    
\end{equation}
where $N \lhd G$ and $K/[K,K]$ is finitely generated. Then there exists $K' \leqslant G'$ such that $K' \lhd G$, $[K,K] \subseteq K'$ and $N \times K'$ has finite index in $G$. 
\end{prop}

Here and below all the direct products of groups (and direct sums of modules) are internal; e.g., by $G'=N \times K$ we mean that $N, K \lhd G'$ and $N \cap K = \{1\}$.

\begin{lem}\label{lem:module_lemma}
Let $M$ be a module over a finite group $S$, with a submodule $Z$, such that
there is a finitely generated subgroup $L \leqslant M$ (not necessarily $S$-invariant) such that $M=Z+L$ and $Z \cap L=\{0\}$. Then there is a finitely generated $S$-submodule $L' \subseteq M$ such that $L'$ is a virtual complement to $Z$ (i.e., $Z \oplus L'$ has finite index in $M$).
\end{lem}

\begin{proof}
Note that
$A = \sum_{s \in S} s.L$
is a finitely generated submodule of $M$, and $B=Z \cap A$ is an $S$-submodule of $A$. For a finite group $S$, every $S$-submodule $B$ in a finitely generated $S$-module $A$ has a submodule that is a  virtual complement: 
see \cite[Lemma~4.2]{Min-virt_retr_props} and its proof (to apply the statement of the lemma directly, one can note that the semidirect product $A \rtimes S$ is a finitely generated virtually abelian group and $B$ is a normal subgroup, hence, by \cite[Lemma~4.2]{Min-virt_retr_props}, $B$ has a normal virtual complement which will be an $S$-submodule of $A$). Let $L'$ be this submodule of $A$, so that
$B \oplus L'$ has finite index in $A$.
Then $Z \cap L' = \{0\}$ and $Z + L'=Z \oplus L'$ has finite index in $Z + A = M$, so $L'$ is an $S$-submodule of $M$ that is a virtual complement to $Z$ in $M$. 
\end{proof}

\begin{proof}[Proof of Proposition~\ref{prop:K'}] Let $Z=Z(N) \lhd G$ denote the center of $N$. Equation \eqref{eq:G'} immediately implies that
\[ZK= C_G(N) \cap G'.\] And since the centralizer $C_G(N)$ and $G'$ are both normal in $G$, we deduce that $ZK=Z\times K \lhd G$. It follows that the derived subgroup $[ZK,ZK]=[K,K]$ is normal in $G$, and we can work in the $G$-module
\[
M = (ZK)/[ZK,ZK]= Z \times K/[K,K].
\]
Since both $N$ and $K$ act trivially on $M$, this is actually an $S$-module, where  $S=G/(NK) = G/G'$ is a finite group. By the assumptions, $L = K/[K,K]$ is a finitely generated complement to $Z$ in $M$ (but $L$ is not necessarily $S$-invariant), so we can apply Lemma~\ref{lem:module_lemma} to find an $S$-submodule $L'$ such that $Z \oplus L'$ has finite index in $M$. Going back to treating $M$ as a $G$-module, we see that $L'$ is normal in $G/[K,K]$, hence its full preimage $K'$, under the homomorphism $G \to G/[K,K]$,
is normal in $G$. Observe that
\begin{enumerate}
    \item[(i)] $ZK'$ has finite index in $ZK$ (because $ZK'$ is the full preimage of $Z + L'$, which has finite index in $M$), and
    \item[(ii)] $Z \cap K' = \{1\}$ in $G$ (because $Z \cap [K,K] = \{1\}$ in $G$ and $Z \cap L' = \{0\}$ in $M$).
\end{enumerate}

From (i), it follows that $ZK'$ contains a finite index subgroup of $K$. Hence
$NK' = NZK'$
has finite index in $NK = G'$. Moreover, $K' \subseteq ZK$, so it centralizes $N$. Note that
\[
N \cap ZK = Z(N \cap K) = Z, \text{ so }
N \cap K' = Z \cap K' = \{1\},
\]
by (ii). Therefore,
$NK' = N \times K'$,
which completes the proof of the proposition. 
\end{proof}

Note that Proposition~\ref{prop:K'} is slightly stronger than Theorem~\ref{thm:normal-in-normal}, because the proposition only requires the abelianization $K/[K,K]$ to be finitely generated.

\subsection{Applications to commensurating graphs of groups}
Below we will be using the notation for graphs of groups from \cite[Subsection~2.2]{minasyan_virtual_2025}.

\begin{defn}\label{def:comm_graphs_of_gps}
We will say that a finite graph of groups     $(\mathcal{G},\Gamma)$ is \emph{commensurating} if the images of every edge group in its two vertex groups have finite indices. More precisely, for each $e \in E\Gamma$, we require that $|G_{\alpha(e)}:\alpha_e(G_e)|<\infty$. 
\end{defn}

Specific examples of commensurating graphs of groups are free products with amalgamation $A*_C B$, where the amalgamated subgroup $C$ has finite index in the factors $A$ and $B$,  and HNN-extensions $A*_{B^t=C}$, where the associated subgroups $B$, $C$ have finite indices in the base groups $A$. Another well-known class of examples is given by finite graphs of groups where all vertex and edge groups are infinite cyclic and whose fundamental groups are often called \emph{generalized Baumslag-Solitar groups}. Commensurating graphs of groups have been originally studied by Bass and Kulkarni in \cite{bass_uniform_1990}, who called them \emph{graphs of groups of finite index}.

\begin{remark} \label{rem:BS-tree}
By Bass-Serre Theory, a group $G$ splits as the fundamental group of a commensurating graphs of groups if and only if $G$ admits a cocompact action on a locally finite  simplicial tree $T$ without edge inversions.    
\end{remark}

\begin{remark}
 Suppose that $(\mathcal{G},\Gamma)$ is a commensurating graph of groups with fundamental group $G$. Definition~\ref{def:comm_graphs_of_gps} easily implies that for any two vertices $u,v \in V\Gamma$, the intersection $G_u \cap G_v$ has finite index in both $G_u$ and $G_v$. Moreover,  if $H \leqslant G$ is the image of any vertex or edge group in $G$ then  $G$ \emph{commensurates} $H$, that is 
\[|H:(H \cap gHg^{-1})|<\infty~\text{ and }~|gHg^{-1}:(H \cap gHg^{-1})|<\infty,~\text{ for all }g \in G.\]
\end{remark}

\begin{defn}\label{def:tame}
A commensurating graph of groups $(\mathcal{G},\Gamma)$ with fundamental group $G$ will be called \emph{tame} (or, more specifically, \emph{tame over $N$})  if there exists a normal subgroup $N \lhd G$ such that the following conditions are satisfied:
\begin{itemize}
    \item[(i)] $N$ is contained as a finite index subgroup in $\alpha_e(G_e)$ in $G$, for each $e \in E\Gamma$;
    \item[(ii)] the natural map $\mathcal{O}: G \to \out(N)$ has finite image.
\end{itemize}
\end{defn}

Examples of tame commensurating graphs of groups are given by \emph{unimodular} Baumslag-Solitar groups (see \cite[Section~2]{Levitt-GBS}), and by commensurating HNN-extensions of $\mathbb{Z}^n$ corresponding to finite order matrices from $\mathrm{GL}(n,\mathbb{Q})$, see \cite{Leary-Min}. 

Observe that the subgroup $N$ in Definition~\ref{def:tame} will have finite index in every vertex group $G_v$, $v \in V\Gamma$, because the graph of groups is commensurating.

\begin{remark} Let $(\mathcal{G},\Gamma)$ be a commensurating graph of groups with fundamental group $G$, and let $T$ be the corresponding locally finite Bass-Serre tree. The action of $G$ on $T$ gives rise to a homomorphism
\[\varphi:G \to \mathrm{Aut}(T),\]
where $\mathrm{Aut}(T)$ is the automorphism group of $T$. Since $T$ is locally finite, $\mathrm{Aut}(T)$ can be naturally topologized, giving it a structure of a locally compact group (see, for example, \cite[Section~3]{bass_uniform_1990}).
The existence of a normal subgroup $N \lhd G$ satisfying condition (i) from Definition~\ref{def:tame} is equivalent to the condition that the image $\varphi(G)$ is a discrete subgroup of $\mathrm{Aut}(T)$ in this topology. The latter amounts to saying that for every vertex $v$ in $T$ the $\varphi(G)$-stabilizer of $v$ is finite (see \cite[Definitions~4.3]{bass_uniform_1990}), and we can define $N=\ker\varphi$.
\end{remark}

The next consequence of Theorem~\ref{thm:normal_virt_comp}, a strong generalization of \cite[Theorem~7.1]{ceccherini-silberstein_multipass_2015}, is the reason why we are interested in the tameness of commensurating graphs of groups.

\begin{cor}\label{cor:emb_of_G_in_virt_N_times_virt_free} Suppose that $G$ is the fundamental group of a commensurating graph of groups $(\mathcal{G},\Gamma)$. If this graph of groups is tame then 
\begin{itemize}
    \item there is a finitely generated free normal subgroup $M \lhd G$ such that $G_v \cap M=\{1\}$ and $|G:G_v M|<\infty$, for all $v \in V\Gamma$; in particular, each $G_v$ embeds as a finite index subgroup in $G/M$;
    \item $G$ embeds as a finite index subdirect product in $G/M \times F$, where $F$ is a finitely generated virtually free group.    
\end{itemize}
\end{cor}

\begin{proof}
Assume that $(\mathcal{G},\Gamma)$ is tame over some $N \lhd G$. Then $N$ will fix every edge of the Bass-Serre tree $T$ for $G$, hence it acts trivially on $T$. Since $N$ has finite index in each vertex group, $F=G/N$ acts on $T$ cocompactly with finite vertex stabilizers, so it is finitely generated and virtually free by the Structure Theorem of Bass-Serre Theory \cite[Section~I.5.4]{serre_trees_2002} and \cite[Theorem~1]{karrass_finite_1973}.

According to the assumptions,  $\mathcal{O}(G)$ has finite index in $\out(N)$, so we can apply \cite[Lemma~5.1]{minasyan_virtual_2025} to conclude that $N$ is a virtual retract of $G$. Now, since $G/N$ is virtually free, we can find a finite index subgroup $H \leqslant G$ such that $N \subseteq H$ and $H/N$ is free. By Theorem~\ref{thm:normal_virt_comp}, there is a finitely generated normal subgroup $M \lhd G$ such that $M \subseteq H$, $M \cap N =\{1\}$ and $|G:MN|<\infty$. It follows that $M$ injects into $H/N$, so it must be free. Since $|G_v:N|<\infty$, for each $v \in V\Gamma$, we see that $|G:G_vM|<\infty$ and $|G_v \cap M|=\{1\}$, as $M$ is torsion-free. The second statement of the corollary can now be deduced similarly to Corollary~\ref{cor:virt-retr-subnormal}.
\end{proof}

The following corollary generalizes Theorem~\ref{thm:main}. The fact that (a) implies (b) is given by Corollary~\ref{cor:emb_of_G_in_virt_N_times_virt_free} (take $Q=G/M$); the proof if the opposite implication is left to the reader.

\begin{cor} Let $G$ be the fundamental group of a commensurating graph of groups $(\mathcal{G},\Gamma)$. Then the following are equivalent:
\begin{itemize}
\item[(a)] $(\mathcal{G},\Gamma)$ is tame;
    \item[(b)] there exists a group $Q$ and a homomorphism $\psi:G \to Q$ such that $\psi$ is injective on each vertex group $G_v$ and $|Q:\psi(G_v)|<\infty$ for all (equivalently, for some) $v \in V\Gamma$.
\end{itemize}
\end{cor}

The next proposition is an immediate consequence of the second claim of Corollary~\ref{cor:emb_of_G_in_virt_N_times_virt_free}.

\begin{prop}\label{prop:prop_P} Let (P) be a property of groups and let $G$ be the fundamental group of a tame commensurating graph of groups $(\mathcal{G},\Gamma)$. Suppose that for some vertex $v \in V\Gamma$ the following conditions hold:
\begin{itemize}
    \item every finite index supergroup of $G_v$ satisfies (P);
    \item every finitely generated virtually free group has (P);
    \item (P) is stable under taking direct products and finite index subgroups.
\end{itemize}
Then $G$ has property (P).
\end{prop}

Proposition~\ref{prop:prop_P} is useful for properties (P) that are not always (or are unknown to be) stable  under commensurability. Hereditary conjugacy separability \cite[Theorem~1.3]{Min-cs_of_fibre_prods} and biautomaticity \cite[Open Question~4.1.5]{Word_processing} are two examples of such properties (recall that a group $G$ is \emph{hereditarily conjugacy separable} if every finite index subgroup is conjugacy separable). The claim about biautomaticity in the next corollary generalizes one direction of \cite[Theorem~8.3]{Leary-Min}.

\begin{cor} Let  $G$ be the fundamental group of a tame commensurating graph of groups with a vertex group $H$. Suppose that $H$ is (word) hyperbolic or finitely generated virtually abelian. Then $G$ is biautomatic. If, additionally, $H$ is virtually compact special (in the sense of Haglund and Wise \cite{HagWise}) then $G$ is hereditarily conjugacy separable.
\end{cor}

\begin{proof} It is well-known that finite index supergroups of hyperbolic groups are hyperbolic;  in particular, this applies to  finitely generated virtually free groups. Hyperbolic groups and finitely generated virtually abelian groups are biautomatic, and biautomaticity is preserved under taking direct products and finite index subgroups \cite{Word_processing}. Therefore, $G$ is biautomatic by Proposition~\ref{prop:prop_P}.

Now, assume that $H$ is virtually compact special, then the same is true for any finite index supergroup of $H$. Virtually compact special hyperbolic groups (which include finitely generated virtually free groups) are hereditarily conjugacy separable by \cite[Theorem~1.1]{Min-Zal}, and finitely generated  virtually abelian groups are hereditarily conjugacy separable by \cite[Proposition~1 in Section~4.C]{Segal-book}. Hereditary conjugacy separability is stable under direct products, by \cite[Lemma~7.3]{Mart-Min}, and under taking finite index subgroups, by definition. Thus Proposition~\ref{prop:prop_P} allows us to conclude that $G$ is hereditarily conjugacy separable.
\end{proof}

\bibliographystyle{alpha} \bibliography{cocommensurability}
\end{document}